\documentclass{article}

\usepackage{microtype}
\usepackage{graphicx}
\usepackage{subfigure}
\usepackage{tikz}
\usepackage{amsmath}
\usepackage{amsthm}
\newtheorem{theorem}{Theorem}
\newtheorem{lemma}[theorem]{Lemma}
\newtheorem*{remark}{Remark}

\usepackage{booktabs} 

\usepackage{hyperref}

\usepackage[accepted]{icml2019}

\icmltitlerunning{Learning to Optimize Multigrid PDE Solvers}

\begin{document}

\twocolumn[
\icmltitle{Learning to Optimize Multigrid PDE Solvers}



\icmlsetsymbol{equal}{*}

\begin{icmlauthorlist}
\icmlauthor{Daniel Greenfeld}{weiz}
\icmlauthor{Meirav Galun}{weiz}
\icmlauthor{Ron Kimmel}{tech}
\icmlauthor{Irad Yavneh}{tech}
\icmlauthor{Ronen Basri}{weiz}
\end{icmlauthorlist}

\icmlaffiliation{weiz}{Weizmann Institute of Science, Rehovot, Israel.}
\icmlaffiliation{tech}{Technion, Israel Institute of Technology, Haifa, Israel}

\icmlcorrespondingauthor{Daniel Greenfeld}{daniel.greenfeld@weizmann.ac.il}

\icmlkeywords{Machine Learning, ICML}

\vskip 0.3in
]


\printAffiliationsAndNotice{}  

\begin{abstract}
\label{abstract}
Constructing fast numerical solvers for partial differential equations (PDEs) is crucial for many scientific disciplines.  A leading technique for solving large-scale PDEs is using multigrid methods. At the core of a multigrid solver is the prolongation matrix, which relates between different scales of the problem. This matrix is strongly problem-dependent, and its optimal construction is critical to the efficiency of the solver. In practice, however, devising multigrid algorithms for new problems often poses formidable challenges. In this paper we propose a framework for learning multigrid solvers. Our method learns a (single) mapping from a family of parameterized PDEs to prolongation operators. We train a neural network once for the entire class of PDEs, using an efficient and unsupervised loss function.  Experiments on a broad class of 2D diffusion problems demonstrate improved convergence rates compared to the widely used Black-Box multigrid scheme, suggesting that our method successfully learned rules for constructing prolongation matrices.
\end{abstract}

\section{Introduction}
\label{Introduction}
Partial Differential Equations (PDEs) are a key tool for modeling diverse problems in science and engineering.
In all but very specific cases, the solution of PDEs requires carefully designed numerical discretization methods, by which the PDEs are approximated by algebraic systems of equations.
Practical settings often give rise to very large ill-conditioned problems, e.g., in predicting weather systems, oceanic flow, image and video processing, aircraft and auto design, electromagnetics, to name just a few.
Developing efficient solution methods for such large systems has therefore been an active research area since many decades ago.

Multigrid methods are leading techniques for solving large-scale discretized PDEs, as well as other large-scale problems (for textbooks see, e.g., \cite{BHM00,TOS01}).
Introduced about half a century ago as a method for fast numerical solution of scalar elliptic boundary-value problems, multigrid methods have since been developed and adapted to problems of increasing generality and applicability.
Despite their success, however, applying off-the-shelf multigrid algorithms to new problems is often non-optimal.
In particular, new problems often require expertly devised prolongation operators, which are critical to constructing efficient solvers.
This paper demonstrates that machine learning techniques can be utilized to derive suitable operators for wide classes of problems.

We introduce a framework for learning multigrid solvers, which we illustrate  by applying the framework to 2D diffusion equations.
At the heart of our method is a neural network that is trained to map discretized diffusion PDEs to prolongation operators, which in turn define the multigrid solver.
The proposed approach has three main attractive properties:\\
\textbf{Scope}.
We train a \textit{single} deep network \textit{once} to handle \textit{any} diffusion equation whose (spatially varying) coefficients are drawn from a given distribution.
Once our network is trained it can be used to produce solvers for any such equation.
Our goal in this paper, unlike existing paradigms, is not to learn to solve a given problem, but instead to learn compact \textit{rules} for constructing solvers for many different problems.\\
\textbf{Unsupervised training}.
The network is trained with no supervision.
It will not be exposed to ground truth operators, nor will it see numerical solutions to PDEs.
Instead, our training is guided by algebraic properties of the produced operators that allude to the quality of the resulting solver.\\
\textbf{Generalization}.
While our method is designed to work with problems of arbitrary size, it will suffice to train our system on quite small problems.
This will be possible due to the local nature of the rules for determining the prolongation operators. 
Specifically, we train our system on block periodic problem instances using a specialized block Fourier mode analysis to achieve efficient training.
At test time we generalize for size (train on $32 \times 32$ grid and test on a $1024 \times 1024$ grid), boundary conditions (train with periodic BCs and test with Dirichlet), and instance types (train on block periodic instances and test on general problem instances).
We compare our method to the widely used Black Box multigrid scheme~\cite{Den82} for selecting operator-dependent prolongation operators, demonstrating superior convergence rates under a variety of scenarios and settings.

\subsection{Previous efforts}

A number of recent papers utilized NN to numerically solve PDEs, some in the context of multigrid methods.
Starting with the classical paper of  \cite{Lagaris98}, many suggested 
 to design a network to solve specific PDEs \cite{ICLR19,Fua18,baymani2010artificial,Berg2018AUD,Han2017OvercomingTC,han2018solving,KDO17,Mishra18,Sirignano18,sun2003solving,Tang17,Wei2018GeneralSF}, generalizing for different choices of right hand sides, boundary conditions, and in some cases to different domain shapes.
These methods require separate training for each new equation.

Some notable approaches in this line of work include \cite{Tang17}, who learn to solve diffusion equations on a fixed grid with variable coefficients and sources drawn randomly in an interval.
A convolutional NN is utilized, and its depth must grow (and it needs to be retrained) with larger grid sizes.
\cite{ICLR19} proposes an elegant learning based approach to accelerate existing iterative solvers, including multigrid solvers. The method is designed for a specific PDE and is demonstrated with the Poisson equation with constant coefficients.
It is shown to generalize to domains which differ from the training domain. 
\cite{Berg2018AUD} handle complex domain geometries by penalizing the PDE residual on collocation points. \cite{han2018solving,Sirignano18} introduce efficient methods for solving specific systems in very high dimensions. \cite{Mishra18} aims to reduce the error of a standard numerical scheme over a very coarse grid. \cite{sun2003solving}  train a neural net to solve the Poisson equation over a surface mesh. 
\cite{Fua18} learn to simulate  computational fluid dynamics to predict the pressures and drag over a surface. \cite{Wei2018GeneralSF} apply deep reinforcement learning to solve specific PDE  instances. 
\cite{KDO17} use a linear NN to derive optimal restriction/prolongation operators for solving a single PDE instance with multigrid. 
The method is demonstrated on 1D PDEs with constant coefficients.
The tools suggested, however, do not offer ways to generalize those choices to other PDEs without re-training.

More remotely, several recent works, e.g., \cite{Chen2019NeuralODE,haber2018learning,chang2018multilevel} suggest an interpretation of neural networks as dynamic differential equations. 
Under this continuous representation, a multilevel strategy is employed to accelerate training in image classification tasks.

\section{Multigrid background and problem setting}
\label{problemDefinition}


We focus on the classical second-order elliptic diffusion equation in two dimensions,
\begin{equation}\label{diffusion}
-\nabla\cdot({\bf g}\nabla {\bf u}) = {\bf f},
\end{equation}
over a square domain, where ${\bf g}$ and ${\bf f}$ are given functions, and the unknown function  ${\bf u}$ obeys some prescribed boundary conditions, for example, Dirichlet boundary conditions whereby ${\bf u}$ is given at every point on the boundary.
The equation is discretized on a square grid of $n \times n$ grid cells with uniform mesh-size $h$.
The discrete diffusion coefficients $g$ are defined at cell centers, while the discrete solution vector  $u$ and the discrete right-hand side vector $f$ are located at the vertices of the grid, as illustrated in the   $3\times 3$ sub-grid depicted in Fig. \ref{fig:subgrid}.

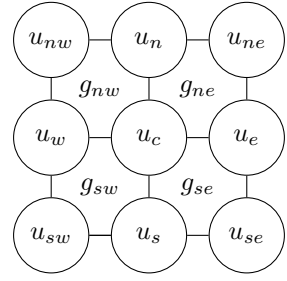
\begin{figure}
\centering
\begin{tikzpicture}[scale=0.65]
    \node[shape=circle,draw=black,minimum size=1cm] (left) at (-2,0) {$u_w$};
    \node[shape=circle,draw=black,minimum size=1cm] (right) at (2,0) {$u_e$};
    \node[shape=circle,draw=black,minimum size=1cm] (center) at (0,0) {$u_c$};
    \node[shape=circle,draw=black,minimum size=1cm] (down_left) at (-2,-2) {$u_{sw}$};
    \node[shape=circle,draw=black,minimum size=1cm] (down_right) at (2,-2) {$u_{se}$};
    \node[shape=circle,draw=black,minimum size=1cm] (down) at (0,-2) {$u_s$};
    \node[shape=circle,draw=black,minimum size=1cm] (top_left) at (-2,2) {$u_{nw}$};
    \node[shape=circle,draw=black,minimum size=1cm] (top_right) at (2,2) {$u_{ne}$};
    \node[shape=circle,draw=black,minimum size=1cm] (up) at (0,2) {$u_n$};
    \node[] (q) at (1,1) {$g_{ne}$};
    \node[] (q) at (-1,-1) {$g_{sw}$};
    \node[] (q) at (1,-1) {$g_{se}$};
    \node[] (q) at (-1,1) {$g_{nw}$};

    \path [-] (left) edge (center);
    \path [-] (left) edge (top_left);
    \path [-] (center) edge (right);
    \path [-] (top_right) edge (right);
    \path [-] (top_right) edge (up);
    \path [-] (down_right) edge (right);
    \path [-] (down_right) edge (down);
    \path [-] (center) edge (up);
    \path [-] (top_left) edge (up);
    \path [-] (down_left) edge (down);
    \path [-] (down_left) edge (left);
    \path [-] (center) edge (down);
    \path [-] (down_left) edge (down);

\end{tikzpicture}
\caption{{\small Sub-grid of $3 \times 3$. The discrete diffusion coefficients $g$ are defined at cell centers. 
The discrete solution $u$ and the discrete right hand side  $f$ are located at the vertices of the grid. 
The discrete equation for  $u_c$ has nine non-zero coefficients multiplying the unknowns  $u_c$ and its eight neighbors.  }} \label{fig:subgrid}
\end{figure}

Employing bilinear finite element discretization, we obtain the following equation associated with the variable $u_c$,
\begin{equation}
\begin{split}
&-\frac{1}{3h^2}(g_{nw}u_{nw}+g_{ne}u_{ne}+g_{se}u_{se}+g_{sw}u_{sw}) \\ &-\frac{1}{6h^2}\big((g_{nw}+g_{ne})u_{n}+(g_{ne}+g_{se})u_{e}+ \\ &(g_{se}+g_{sw})u_s+(g_{sw}+g_{nw})u_w \big)\\
&+\frac{2}{3h^2}(g_{nw}+g_{ne}+g_{se}+g_{sw})u_c = f_c \, .
\end{split}
\end{equation}
\newline
Arranging these equations in matrix-vector form, we obtain a linear system
\begin{equation}\label{discretization}
Au=f ,
\end{equation}
where $A_{c,j}$ is the coefficient multiplying $u_j$ in the discrete equation associated with $u_c$.
The term ``the stencil of $u_c$" will refer to the $ 3 \times 3$ set of coefficients associated with the equation for $ u_c $.

The discretization matrix $A$ is symmetric positive semi-definite (and strictly positive definite in the case of Dirichlet boundary conditions) and sparse, having at most nine non-zero elements per row, corresponding to the nine stencil elements. The size of $u$, i.e., the number of unknowns, is approximately $n^2$ (with slight variations depending on whether or not boundary values are eliminated), while the size of $A$ is approximately $n^2 \times n^2$.
For large $n$, these properties of $A$ render iterative methods attractive. One simple option is the classical Gauss-Seidel relaxation, which is induced by the splitting $A = L + U$, where $L$ is the lower triangular part of $A$, including the diagonal, and $U$ is the upper triangular part of $A$.
The resulting iterative scheme,
\begin{equation} \label{Gauss_Seidel}
u^{(k)} = u^{(k-1)} + L^{-1} \left(f-A u^{(k-1)} \right) \, ,
\end{equation}
is convergent for symmetric positive definite matrices.
Here, $(k)$ denotes the iteration number.
The error after iteration $k$, $e^{(k)} = u - u^{(k)}$, is related to the error before the iteration by the error propagation equation,
\begin{equation} \label{Gauss_Seidel_Error_Propagation}
e^{(k)} = S e^{(k-1)} \, ,
\end{equation}
where $S = I - L^{-1}A$ is the error propagation matrix of the Gauss-Seidel relaxation, with $I$ denoting the identity matrix of the same dimension as $A$.

Although the number of elements of $A$ is $O(n^4)$, Gauss-Seidel iteration requires only $O(n^2)$ arithmetic operations because $A$ is extremely sparse, containing only $O(n^2)$ nonzero elements. Nevertheless, as a stand-alone solver Gauss-Seidel is very inefficient for large $n$ because the matrix $A$ is highly ill-conditioned resulting in slow convergence.
However, Gauss-Seidel is known to be very efficient for \textit{smoothing} the error.
That is, after a few Gauss-Seidel iterations, commonly called relaxation sweeps, the remaining error varies slowly relative to the mesh-size, and it can therefore be approximated well on a coarser grid. This motivates the multigrid algorithm, which is described next.

\subsection{Multigrid Cycle}
\label{TwoGrid}

A coarse grid is defined by skipping every other mesh point in each coordinate, obtaining  a grid of $\frac{n}{2} \times \frac{n}{2}$ grid cells and mesh-size $2h$.
A prolongation operator $P$ is defined and it can be represented as a sparse matrix whose number of rows is equal to the size of $u$ and the number of columns is equal to the number of coarse-grid variables, approximately $\left(\frac{n}{2}\right)^2$. 
The two-grid version of the multigrid algorithm proceeds by applying one or more relaxation sweeps on the fine grid, e.g., Gauss-Seidel, obtaining an approximation $\tilde{u}$ to $u$, such that the remaining error, $u - {\tilde u}$ is smooth and can therefore be approximated well on the coarse grid. 
The linear system for the error is then projected to the coarse grid by the Galerkin method as follows.  
The coarse grid operator is defined as  $P^TAP$ and the right-hand-side is the restriction of the residual to the coarse grid, i.e., $P^T (f-A \tilde{u})$. 
Then, the coarse-grid system is solved directly in the two-grid algorithm, recursively in multigrid, and the resulting solution is transferred by the prolongation $P$ to the fine grid and added to the current approximation. 
This is typically followed by one or more additional fine-grid relaxation sweeps. 
This entire process comprises a single two-grid iteration as formally described in Algorithm \ref{alg:MG}.
\begin{algorithm}
   \caption{Two-Grid Cycle}
   \label{alg:MG}
\begin{algorithmic}[1]
   \STATE {\bfseries Input:} Discretization matrix $A$, initial approximation $u^{(0)}$, right-hand side $f$, prolongation matrix $P$,  a relaxation scheme, $k = 0$, residual tolerance $\delta$
   \REPEAT
    \STATE Perform $s_1$ relaxation sweeps starting with the current approximation $u^{(k)}$, obtaining ${\tilde u}^{(k)}$
    \STATE Compute the residual: $r^{(k)}=f-A {\tilde u}^{(k)}$
    \STATE Project the error equations to the coarse grid and solve the coarse grid system: $P^T A P v^{(k)} = P^T r^{(k)}$
    \STATE Prolongate and add the coarse grid solution: ${\tilde u}^{(k)}={\tilde u}^{(k)}+P v^{(k)}$
    \STATE Perform $s_2$ relaxation sweeps obtaining $u^{(k+1)}$
    \STATE $k = k+1$
    \UNTIL{$r^{(k-1)}<\delta$}
\end{algorithmic}
\end{algorithm}

In the multigrid version of the algorithm, Step 5 is replaced by one or more recursive calls to the two-grid algorithm, employing successively coarser grids.
A single recursive call yields the so-called multigrid V cycle, whereas two calls yield the W cycle. These recursive calls are repeated until reaching a
very coarse grid, where the problem is solved cheaply by
relaxation or an exact solve. The entire multigrid cycle thus
obtained has linear computational complexity. The W cycle
is somewhat more expensive than the V cycle but may
be cost-effective in particularly challenging problems.  
 
The error propagation equation of the two-grid algorithm is given by
\begin{equation} \label{Two_Grid_Error_Propagation}
e^{(k)} = M e^{(k-1)},
\end{equation}
where $M=M(A,P;S,s_1,s_2)$ is the two-grid error propagation matrix
\begin{equation}\label{eq:M}
M = S^{s_2} C  S^{s_1}.
\end{equation}
Here, $s_1$ and $s_2$ are the number of relaxation sweeps performed before and after the coarse-grid correction phase, and the error propagation matrix of the coarse grid correction is given by
\begin{equation}\label{eq:C}
C = (I-P \left[ P ^T A P \right]^{-1} P^T A).
\end{equation}
For a given operator $A$, the error propagation matrix $M$ defined in \eqref{eq:M} governs the convergence behavior of the  two-grid (and consequently multigrid) cycle. 
The cycle efficiency relies on the complementary roles of the relaxation $S$ and the coarse-grid correction $C$; that is, the error propagation matrix of the coarse grid correction phase, $C$, must reduce significantly any error which is not reduced by $S$, called \textit{algebraically smooth error}.

For symmetric positive definite $A$ and full-rank $P$, as we assume throughout this discussion, the matrix $P \left[ P ^T A P \right]^{-1} P^T A$ in \eqref{eq:C} is an $A$-orthogonal projection onto the range of $P$ (i.e., the subspace spanned by the columns of $P$). 
Thus, $C$, the error propagation matrix of the coarse grid correction phase \eqref{eq:C}, essentially subtracts off the component of the error that is in the range of $P$.   
This requires that the algebraically smooth error will approximately be in the range of $P$. 
The task of devising a good prolongation is challenging, because $P$ also needs to be very sparse for computational efficiency.

Commonly, a specific relaxation scheme, such as Gauss-Seidel, is preselected, as are the number of relaxation sweeps per cycle, and therefore the efficiency of the cycle is governed solely by the prolongation operator $P$. The challenging task therefore is to devise effective prolongation operators. A common practice for diffusion problems on structured grids is to impose on $P$ the sparsity pattern of bilinear interpolation\footnote{Assume that $u_c$ in the subgrid diagram  in Fig.~\ref{fig:subgrid} coincides with a coarse-grid point $U_c$. Then the column of $P$ corresponding to $U_c$ contains nonzero values only at the rows corresponding to the nine fine-grid variables appearing in the diagram.} and then to skillfully select values of the nonzero elements of $P$ based locally on the elements of the discretization matrix $A$. 
In contrast, our approach is to automatically learn the local rules for determining the prolongation coefficients by training a single neural network, which can be applied  to  the entire class of diffusion equations discretized by $3 \times 3$ stencils.

\section{Method}
\label{Method}

We propose a scheme for learning a mapping from discretization matrices to prolongation matrices. We assume that the diffusion coefficients are drawn from some distribution, yielding a distribution $\mathcal{D}$ over the discretization matrices. A natural objective  would be to seek a mapping that minimizes the expected spectral radius of the error propagation matrix $M(A,P)$ defined in $\eqref{eq:M}$, which governs the asymptotic convergence rate of the multigrid solver. Concretely, we represent the mapping with a neural network parameterized by $\theta$ that maps discretization matrices $A$ to prolongations $P_\theta(A) \in {\cal P}$ with a predefined sparsity pattern. The relaxation scheme $S$ is fixed to be Gauss-Seidel, and the parameters $s_1, s_2$ are set to 1. Thus, we arrive at the following learning problem: 
\begin{equation} \label{eq:loss} 
\min_{P_\theta \in \mathcal{P}} \mathbf{E}_{A\sim\mathcal{D}}
~~\rho(M(A,P_\theta(A))), 
\end{equation}
where $\rho(M)$ is the spectral radius of the matrix $M$, and $\mathcal{D}$ is the distribution over the discretization matrices $A$.

\begin{figure}
\begin{center}
\centerline{\includegraphics[width=0.4\columnwidth]{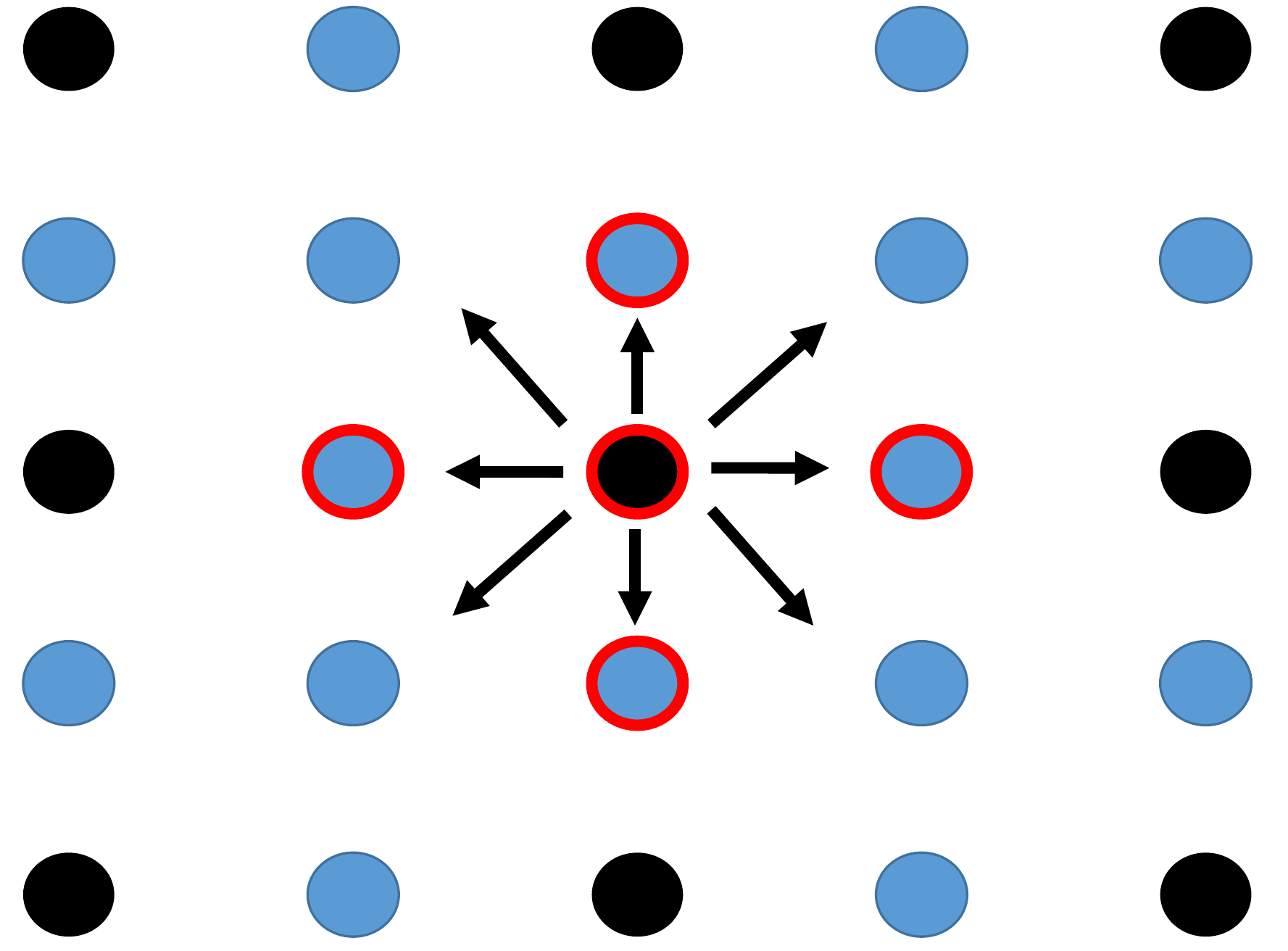}}
\caption{{\small The input and the output of  the network. The discs denote the (fine) grid points, where the black discs mark the subset of points selected as coarse grid points. The input of the network consists of the $ 3 \times 3$ stencils of the five points, denoted by the red cycles. The black arrows illustrate the output of the network, i.e., the  contribution of the prolongation of one coarse point to its eight fine grid neighbors.}}
\label{fig:NetInput}
\end{center}
\end{figure}

\subsection{Inferring $P$ from local information} \label{Inferring_P}

The network we construct receives an input vector of size 45, consisting of a local subset of the discretization matrix $A$, and produces an output that consists of 4 numbers, which in turn determine the 9 nonzero entries of one column of the prolongation matrix $P$. Existing multigrid solvers for diffusion problems on structured grids (e.g., \cite{ABDP81,Zee90,Den82}), infer the prolongation weights from local information. Following their approach, we construct our network to determine each column $j$ of $P$ from five $3 \times 3$ stencils. Specifically, the input to the network is composed of the stencil of the fine grid point coinciding with coarse point $j$, and the stencils of its four immediate neighbors, marked by the red circles in Fig.~\ref{fig:NetInput}.

For the output we note that the sparsity pattern imposed on $P$ implies that each column has at most nine non-zero elements, where  each non-zero element $P_{ij}$ is the prolongation weight of the coarse grid point $j$ to a nearby fine grid point $i$. Geometrically, this means that a coarse grid point contributes only to the fine-grid point with which it coincides (and the corresponding prolongation coefficient is set to 1) and to its eight fine-grid neighboring points, as illustrated in Fig.~\ref{fig:NetInput}. Only the four prolongation coefficients corresponding to the nearest neighbors are learned; the four remaining prolongation coefficients, marked by diagonal arrows in Fig.~\ref{fig:NetInput}, are then calculated  such that any grid function $u$ obtained by prolongation from the coarse grid satisfies $Au=0$ at these four grid points. The complete prolongation matrix $P$ is constructed by applying the same network repeatedly to all the coarse points.

The inference from local information maintains the efficiency of the resulting multigrid cycle, as the mapping has constant time computation per coarse grid point, and we construct $P$ by applying the network repeatedly to all coarse grid points. Moreover, the local nature of the inference allows application of the network on different grid-sizes. Further details are provided in Section \ref{Experiments}.

\subsection{Fourier analysis for efficient training}
The fact that the network determines $P$ locally does not mean that it suffices to train on very small grids. Because the method is to be used for large problems, it is critical that the subspace spanned by the columns of $P$ will approximate well all algebraically smooth errors of large problems, as discussed, e.g., in \cite{RDF06}. This implies that such errors should be encompassed in the loss function of the  training phase. In practice, our experiments show that good performance on large grids is already obtained after training only on a $32\times32$ grid, which is not very large but still results in an error propagation matrix $M$\ of size $1024\times1024$.

The main computational barrier of the loss \eqref{eq:loss} is due to the coarse-grid correction matrix $C$ \eqref{eq:C}, whose computation requires inversion of the matrix $P^T A P$ of size  $(n/2)^2 \times (n/2)^2$ elements. To overcome this prohibitive computation, we introduce two surrogates. First, we relax the spectral radius of the error propagation matrix with its squared Frobenious norm, relying on the fact that the Frobenious norm bounds the spectral radius from above, yielding a differentiable quantity without the need for (expensive) spectral decomposition. Secondly, we train on a  relatively limited class of discretization matrices,  $A$, which are called block-circulant matrices, allowing us to train efficiently on large problems, because it requires inversion only of small matrices, as explained below.
Due to   the local dependence of $P$ on $A$, we expect that the resulting trained network would be equally effective for general (non block-periodic) $A$, and this is indeed borne out in our experiments.

The block-periodic framework allows us to train efficiently on large problems. To do so, we exploit a block Fourier analysis technique that was recently introduced independently in several variants and for different applications \cite{BR18,BHM18,KRGO18}. Classical Fourier analysis has been employed for quantitative prediction of two-grid convergence factors since the 1970s. This technique, however, is exact only in very special cases of constant-coefficient operators and simple boundary conditions. Here, in contrast, we need to cater to arbitrary discrepancies in the values of the diffusion coefficients of neighboring grid cells, which imply strongly varying coefficients in the matrix $A$, so classical Fourier analysis is not appropriate.

To apply the new block Fourier analysis, we partition our $n \times n$ grid into equal-sized square blocks of $c \times c$  cells each, such that all the $\frac{n}{c} \times \frac{n}{c}$ blocks are identical with respect to their cell $g$ values, but within the block the $g$ values vary arbitrarily, according to the original distribution. This can be thought of as tiling the domain by identical blocks of $c \times c$ cells. Imposing periodic boundary conditions, we obtain a discretization matrix $A$ that is block-circulant. Furthermore, due to the dependence of $P$ on $A$, the matrix $M$ itself is similarly block-circulant and can be written as
\begin{equation} \label{eq:BlockCirculant}
M=\begin{bmatrix}
M_0 & M_1 & \ldots & M_{b-2} & M_{b-1} \\
M_{b-1} & M_0 & M_1 & \ldots & M_{b-2} \\
M_{b-2} & M_{b-1} & M_0 & \ldots & M_{b-3} \\
\ldots & \ldots & \ldots & \ldots & \ldots \\
M_1 & \ldots & M_{b-2} & M_{b-1} & M_0
\end{bmatrix} ,
\end{equation}
where $M_j, j=0, \ldots, b-1$, are $c^2 \times c^2$ blocks and $b = \frac{n^2}{c^2}$. This special structure has the following important implication. $M$ can easily be block-diagonalized in a way that each block of size $c^2 \times c^2$ on the diagonal has a simple closed form that depends on the elements of $A$ and a single parameter associated with a certain Fourier component. As a result, the squared Frobenius norm of the matrix $M$, which constitutes the loss for our network, can be decomposed into a sum of squared Frobenius norms of these small easily computed blocks, requiring only the inversion of relatively small matrices.

The theoretical foundation of this essential tool is summarized briefly below.  For further details, we refer the reader to the supplemental material and to \cite{BR18,BHM18}.

\paragraph{Block diagonalization of block circulant matrices}
Let the $n \times n$ matrix $K$ be block-circulant, with $b$ blocks of size $k$. That is, $n = bk$, and the elements of $K$ satisfy:
\begin{equation} \label{eq:BlockPeriodic}
K_{l,j} = K_{{\rm mod} (l-k,n), {\rm mod} (j-k,n)},
\end{equation}
with rows, column, blocks, etc., numbered starting from 0 for convenience. Here, we are adopting the MATLAB form mod$(x,y) =$  ``$x$ modulo $y$'', i.e., the remainder obtained when dividing integer $x$ by integer $y$. Below, we continue to use $l$ and $j$ to denote row and column numbers, respectively, and apply the decomposition:
\begin{equation} \label{eq:Decomposition}
l = l_0 + tk,~~~~j = j_0 + sk \, ,
\end{equation}
where $l_0 = {\rm mod}(l,k)$, $t = \lfloor \frac{l}{k} \rfloor$, $j_0 = {\rm mod}(j,k)$, $s = \lfloor \frac{j}{k} \rfloor$. Note that $ l,j \in \{0,...,n-1\}$; $l_0,j_0 \in \{0,...,k-1\}$; $t,s \in \{0,...,b-1\} \, .$

\noindent
Let the column vector
$$
v_m = \left[ 1,e^{i\frac{2\pi m}{n}}, \ldots, e^{i\frac{2\pi mj}{n}}, \ldots, e^{i\frac{2\pi m(n-1)}{n}} \right]^*
$$
denote the unnormalized $m$th Fourier component of dimension $n$, where $m=0,\ldots,n-1$. Finally, let $W$ denote the $n \times n$ matrix whose nonzero values are comprised of the elements of the first $b$ Fourier components as follows:
\begin{equation} \label{eq:W}
W_{l,j} = \frac{1}{\sqrt{b}} \delta_{l_0,j_0} v_s(l) \, ,
\end{equation}
where $v_s(l)$ denotes the $l$th element of $v_s$, and $\delta$ is the Kronecker delta. Then we have:
\begin{theorem}  \label{thm:1}
$W$ is a unitary matrix, and the similarity transformation ${\hat K} = W^* K W$ yields a block-diagonal matrix with $b$ blocks of size $k \times k$, ${\hat K} = \mathrm{blockdiag} \left( {\hat K}^{(0)}, ...,  {\hat K}^{(b-1)} \right)$.  Furthermore, if $K$ is band-limited modulo $n$ such that all the nonzero elements in the $l$th row of $K$, $l = 0, ..., n-1$, are included in $\{K_{l,{\rm mod}(l-\alpha,n)}, ..., K_{l,l}, ..., K_{l, {\rm mod}(l+\beta,n)} \}$, and $\beta + \alpha + 1 \leq k$, then the nonzero elements of the blocks are simply 
\begin{eqnarray*}
{\hat K}_{l_0, {\rm mod} (l_0+m,k)}^{(s)} & = & e^{-i \frac{2\pi sm}{n}} K_{l_0,{\rm mod}(l_0+m,n)} \, , \\
l_0 & = & 0, ..., k-1, ~~ m = -\alpha, ..., \beta \, .
\end{eqnarray*}
\end{theorem}
\textit{The proof is in the supplementary material.}

By applying  Theorem \ref{thm:1} recursively, we can block diagonalize $M$~\eqref{eq:BlockCirculant} for our 2D problems.  

In practice,  for computational efficiency, we perform an equivalent analysis   using Fourier symbols for each of the multigrid components as is commonly done in multigrid Fourier analysis  (see, e.g., \cite{wienands2004practical}).  We finally compute the loss
\[ \|M\|^2_{F} = \|\hat M\|^2_F = \sum_{s=0}^{b-1} \| \hat M^{(s)} \|^2_F, \]
where $\hat M=\mathrm{blockdiag}\left(\hat M^{(0)},...,\hat M^{(b-1)}\right)$. Note that, $\| \hat M^{(s)} \|^2_F$ is cheap to compute since $\hat M^{(s)}$ is of size $c^2 \times c^2$ ($c=8$ in our experiments).

To summarize, Theorem~\ref{thm:1} allows us to train on block- periodic\ problems with grid size of $n \times n$ using $\frac{n^2}{c^2}$ matrices of size $c^2 \times c^2$ instead of a matrix of size $n^2 \times n^2$.


\section{Experiments}
\label{Experiments}

For evaluating our algorithm several measures are employed, and we compare the performance of our network based solver to the classical and widely used Black Box multigrid scheme \cite{Den82}. To the best of our knowledge, this is the most efficient scheme for prolongation construction for diffusion problems. We train and test the solver for the diffusion coefficients $g$ sampled from a log-normal distribution, which is commonly assumed, e.g., in modeling flow in porous media (cf. \cite{MDH98}), where Black Box prolongation is used for homogenization in this regime). As explained above, the network is trained to minimize the Frobenious norm of the error propagation matrix of rather small grids comprised of circulant blocks and periodic boundary conditions. However, the tests are performed for a range of grid sizes, general non block-periodic $g$, Dirichlet boundary conditions, and even a different domain. Finally, we remark that the run-time per multigrid cycle of the network based algorithm is the same as that of Black Box multigrid, due to the identical sparsity pattern. However, the once-per-problem setup phase of the network based algorithm is more expensive than that of Black Box scheme because the former uses the trained network to determine $P$ whereas the latter uses explicit formulas.

\paragraph{Network details}
The inputs and outputs to our network are specified in Sec. \ref{Inferring_P}. We train a residual network consisting of $100$ fully-connected layers of width $100$ with RELU activations. Note that all matrix-dependent multigrid methods, including Black-Box, apply local nonlinear mappings to determine the prolongation coefficients.  

\paragraph{Handling the singularity}
Employing block Fourier analysis, as we do for efficiency, requires training with periodic boundary conditions. This means that our discretization matrices $A$ are singular, with null space comprised of the constant vector. This in turn means that $P^TAP$ is also singular and cannot be inverted, so $M$ cannot be explicitly computed. We overcome this problem by taking two measures. First, we impose that the sum of each row of $P$ be equal to 1. This ensures that the null space of the coarse-grid matrix $P^TAP$ too is comprised of the (coarse-grid) constant vector. Second, when computing the loss with the block Fourier analysis, we ignore the undefined block which corresponds to the zeroth Fourier mode (i.e., the constant vector). To force the rows of the prolongation to sum to one, we simply normalize the rows of $P$ that are learned by the network (left, right, above and below each coarse-grid point) before completing the construction of $P$ as described in Section~\ref{Inferring_P}.
When dealing with Dirichlet boundary conditions, this constraint is not feasible for rows corresponding to points near the boundary. For those points, we use the prolongation coefficients proposed by the Black Box algorithm.

\paragraph{Training details}
Training is performed in three stages. First, the network was trained for two epochs on $163840$ diffusion problems with grid-size $16 \times 16$ composed of $8\times 8$ doubly-periodic core blocks and with doubly periodic boundary conditions. This results in an tentative network, which is further trained as follows. The tentative network was used to create prolongation matrices for $163840$ non block-periodic diffusion problems with grid-size $16\times 16$ and periodic boundary conditions. Then, using Galerkin coarsening $P^TAP$, this resulted in $163840$ $8 \times 8$ blocks corresponding to coarse level blocks, which were used as core blocks for generating $16 \times 16$ block periodic problems.   Now, at the second stage, the new training set which consists of $2 \times 163840 $ problems, was used for additional two epochs. After that, at the last stage, those $8 \times 8$ core blocks were used to compose problems of grid-size $32 \times 32$, and the training continued for two additional epochs. The second stage was done to facilitate good performance on coarse grids as well, since in practice a two grid scheme is too expensive and recursive calls are made to solve the coarse grid equation. The network was initialized using the scheme suggested in \cite{zhang2018residual}. Throughout the training process, the optimizer used was Adam, with an initial learning rate drawn from $10^{-U([4,6])}$.

\subsection{Evaluation}

\paragraph{Spectral radius}
As a first evaluation, we present the spectral radius of the two-grid error propagation matrix obtained with our network on $64 \times 64$ grid problems with Dirichlet boundary conditions, where the diffusion coefficients were drawn from a log-normal distribution. Table~\ref{spectral_radius_table} shows the results, averaged over 100 instances. We observe that the network based algorithm clearly outperforms Black Box multigrid by this measure, achieving a lower average $\rho(M)$, despite the discrepancies between the training and testing conditions (block-periodic $g$, Frobenius norm minimization and smaller grid in the training, versus general $g$, Dirichlet boundary conditions, spectral radius and larger grid in the tests).

\begin{table}[t]
\caption{{\small Spectral radius of the two-grid error propagation matrix $M$ for a $64 \times 64$ grid with Dirichlet boundary conditions (smaller is better).}}
\label{spectral_radius_table}
\vskip 0.15in
\begin{center}
\begin{small}
\begin{sc}
\begin{tabular}{lcccr}
\toprule
Method & Spectral radius\\
\midrule
Black Box    & $0.1456 \pm 0.0170$\\
Network      & $0.1146 \pm 0.0168$\\

\bottomrule
\end{tabular}
\end{sc}
\end{small}
\end{center}
\vskip -0.1in
\end{table}

\paragraph{Multigrid cycles}
Numerical experiments are performed with V and W cycles. In each experiment, we test $100$ instances with Dirichlet boundary conditions, and the diffusion coefficients in each instance are drawn from a log-normal distribution. We solve the homogenous problem $Au=0$, with the initial guess for the solution drawn from a normal distribution\footnote{Due to the linearity of the problem and the algorithm, the convergence behavior is independent of $f$ and of the Dirichlet boundary values; we choose the homogeneous problem in order to allow us to run many cycles and reach the worst-case asymptotic regime without encountering roundoff errors when the absolute error is on the order of machine accuracy.}. In each experiment we run $40$ multigrid cycles and track the error norm reduction factor per cycle, $\frac{||e^{(k+1)}||_2}{||e^{(k)}||_2}$. We consider the ratio in the final iteration to be the asymptotic value.

Figure \ref{log_normal_w_cycle_1024} (left) shows the norm of the error as a function of the iteration number for a W cycle, where the fine grid-size is $1024 \times 1024$ and nine grids are employed in the recursive multigrid hierarchy. Both algorithms exhibit the expected fast multigrid convergence. Figure \ref{log_normal_w_cycle_1024} (right) shows the error reduction factor per iteration for this experiment. We see that the mean convergence rates increase with the number of iterations but virtually level off at asymptotic convergence factors of about $0.2$ for Black Box multigrid and about $0.16$ for the network-based method.

\begin{figure}
\centering
\includegraphics[width=0.5\columnwidth]{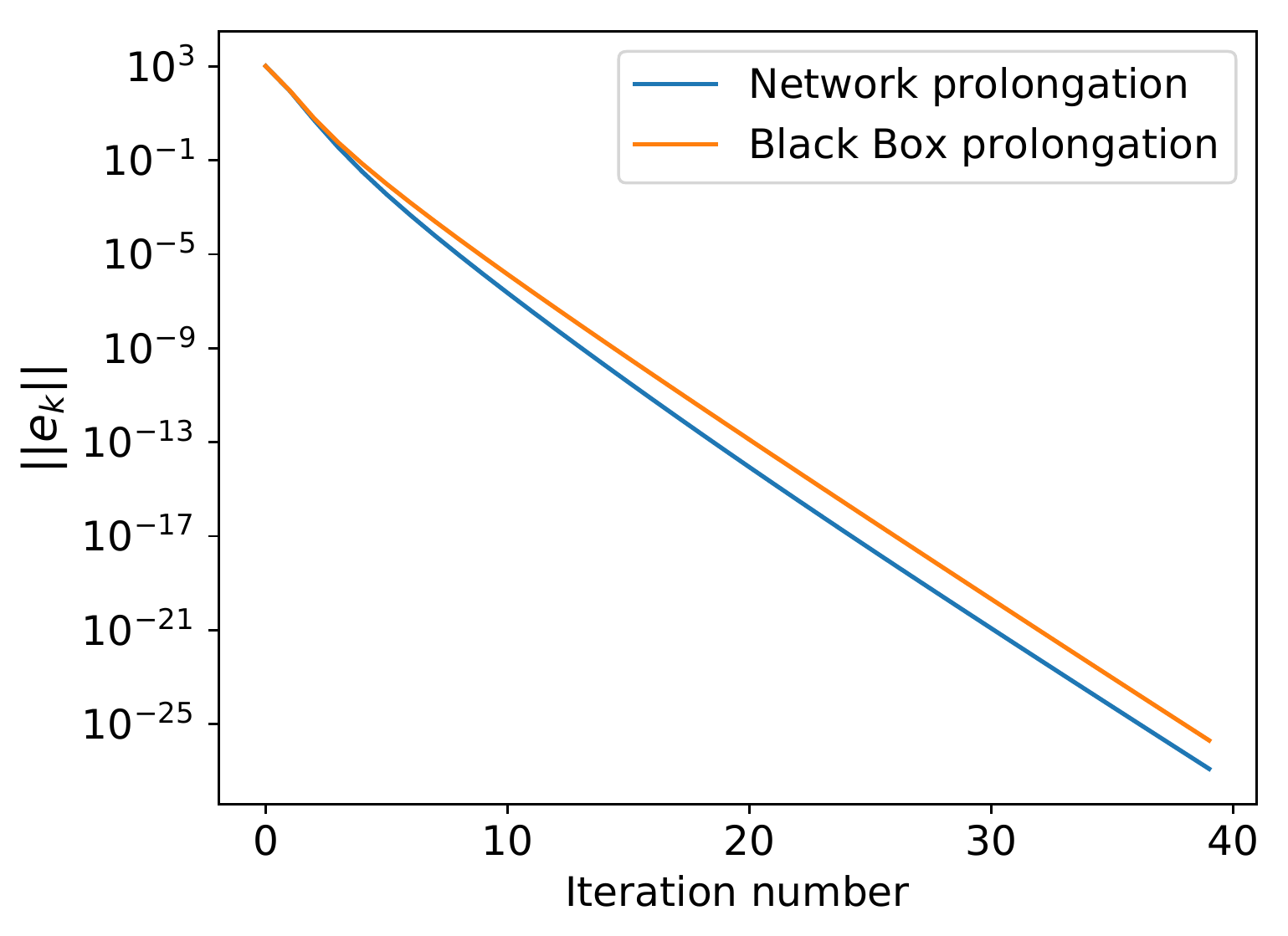}~~~~
\includegraphics[width=0.5\columnwidth]{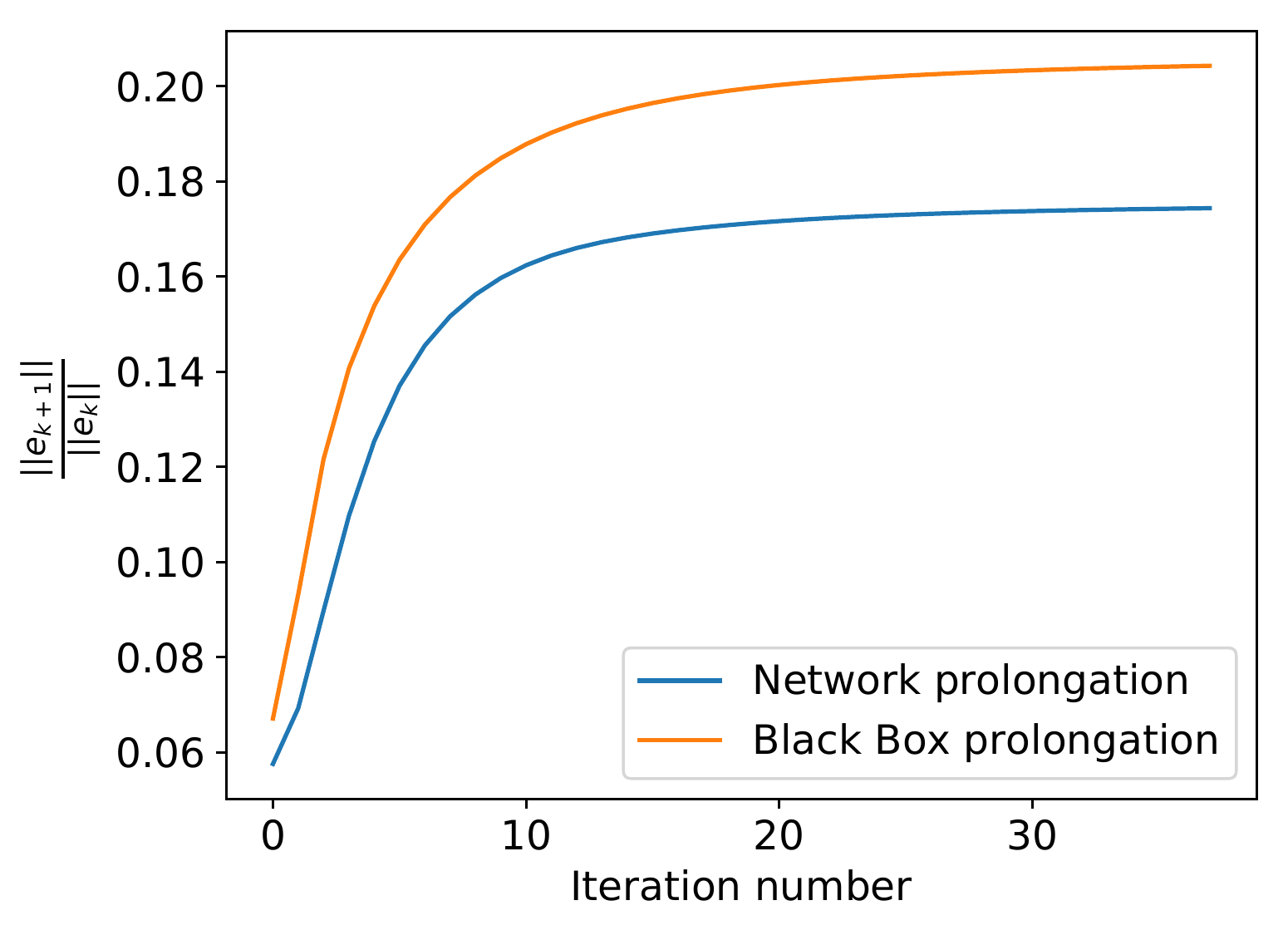}
\caption{{\small W-cycle performance, averaged over $100$ problems with grid size $1024\times 1024$ and Dirichlet Boundary conditions. Left: error norm as a function of iterations (W cycles). Right: error norm reduction factor per iteration.}}
\label{log_normal_w_cycle_1024}
\end{figure}

\begin{figure}[ht]
\begin{center}
\centerline{\includegraphics[width=0.6\columnwidth]{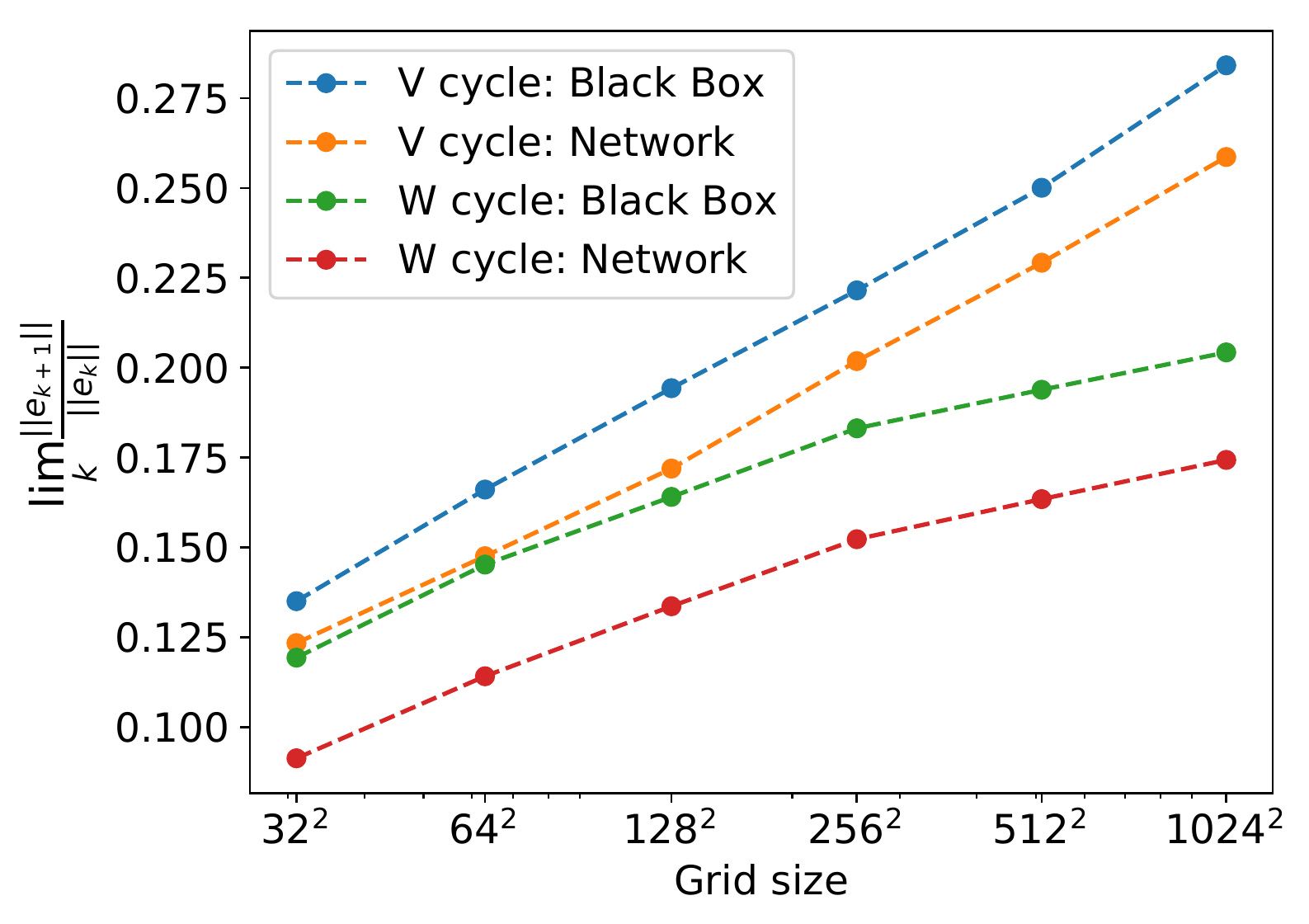}}
\caption{{\small V cycle and W cycle average asymptotic error norm reduction factor per iteration.}}
\label{log_normal_v_w_cycle}
\end{center}
\vskip -0.2in
\end{figure}

Figure \ref{log_normal_v_w_cycle} shows the asymptotic error norm convergence factors per cycle of V and W cycles with fine-grid sizes ranging from $32 \times 32$ to $1024 \times 1024$.
Additionally, Table \ref{success_table_log_normal} shows the success rate of the network based method, defined as the percentage of instances in which it outperformed the Black Box algorithm in terms of asymptotic convergence factor. Evidently, the network based method is superior by this measure, and we see no significant deterioration for larger grids, even though the training was performed on relatively small grids and with block-periodic $g$.
\begin{table}[t]
\caption{{\small Success rate of V cycle and W cycle with log-normal $g$ distribution.}}
\label{success_table_log_normal}
\vskip 0.15in
\begin{center}
\begin{small}
\begin{sc}
\begin{tabular}{ccccr}
\toprule
Grid size & V-cycle & W-cycle\\
\midrule
$32 \times 32$ & 83 \%& 100 \%\\
$64 \times 64$ & 92 \%& 100 \%\\
$128 \times 128$ & 91 \% & 100 \%\\
$256 \times 256$ & 84 \%& 99 \%\\
$512 \times 512$ & 81 \%& 99 \%\\
$1024 \times 1024$ & 83 \%& 98 \%\\
\bottomrule
\end{tabular}
\end{sc}
\end{small}
\end{center}
\vskip -0.1in
\end{table}

\paragraph{Uniform distribution}
As a test of robustness with respect to the diffusion coefficient distribution, we evaluate the network trained with log-normal distribution on a different distribution of the $g$ values. Here, we present the results of applying multigrid cycles as in the previous experiment, except that in these tests the diffusion coefficients are drawn from the uniform distribution over $[0,1]$. The results are shown in Figure \ref{uniform_v_w_cycle}, with Table \ref{success_table_uniform}, as before, showing the success rate of the network in these tests. Evidently, the advantage of the network based method is narrower in this case, due to the mismatch of distributions, but it still exhibits superior convergence factors.

\begin{figure}[ht]
\begin{center}
\centerline{\includegraphics[width=0.6\columnwidth]{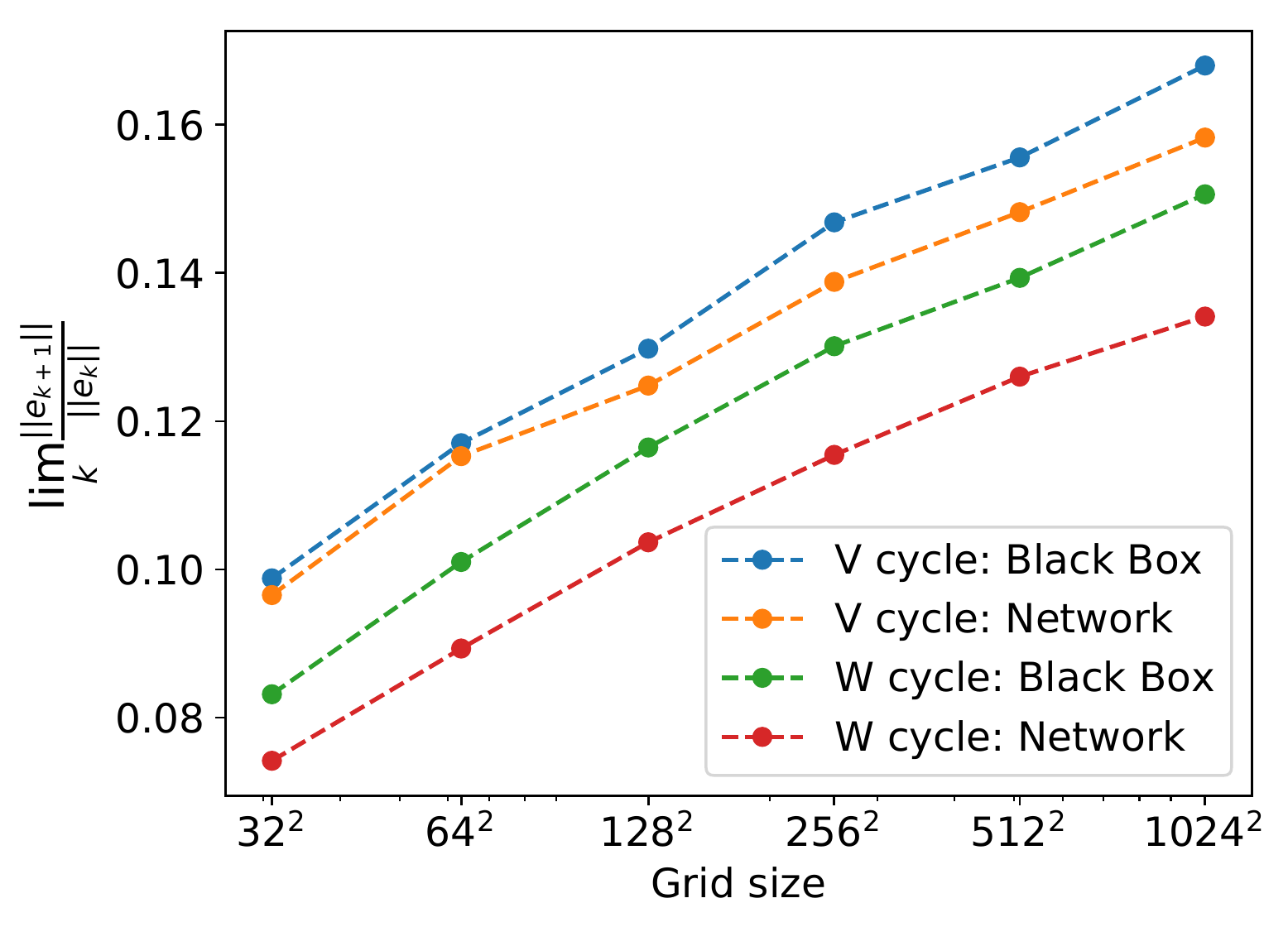}}
\caption{{\small V cycle and W cycle average asymptotic error norm reduction factor per iteration tested with uniform $g$ distribution, with network trained on log-normal distribution.}}
\label{uniform_v_w_cycle}
\end{center}
\vskip -0.2in
\end{figure}

\begin{table}[t]
\caption{{\small Success rate of V cycle and W cycle with uniform $g$ distribution.}}
\label{success_table_uniform}
\begin{center}
\begin{small}
\begin{sc}
\begin{tabular}{ccccr}
\toprule
Grid size & V-cycle & W-cycle\\
\midrule
$32 \times 32$ & 60 \%& 90 \%\\
$64 \times 64$ & 54 \%& 90 \%\\
$128 \times 128$ & 66 \%& 91 \% \\
$256 \times 256$ & 79 \%& 91 \%\\
$512 \times 512$ & 81 \%& 88 \%\\
$1024 \times 1024$ & 81 \%& 96 \%\\
\bottomrule
\end{tabular}
\end{sc}
\end{small}
\end{center}
\vskip -0.1in
\end{table}

\paragraph{Non-square domain}
In the next experiment, we test our network on diffusion problems specified on a domain consisting of a two-dimensional disk.
Our method achieves a better convergence rate in this case too, see Table \ref{tab:Irregular_domain}.
\begin{table}[t]
\caption{{\small Asymptotic error reduction factor per cycle on a $2D$ disk with a diameter of 64 grid points, averaged over 100 instances.}}
\label{Disk}
\begin{center}
\begin{small}
\begin{sc}
\begin{tabular}{lcccr}
\toprule
Method & V-cycle & W-cycle\\
\midrule
Black Box  & $0.1969 \pm 0.0290$ &$0.1639 \pm 0.0169$\\
Network    & $0.1868 \pm 0.0296$ &$0.1352 \pm 0.0155$\\

\bottomrule
\end{tabular}
\end{sc}
\end{small}
\end{center}
\vskip -0.1in
 \label{tab:Irregular_domain}
\end{table}

\paragraph{Diagonally dominant problems}
In the final experiment, we evaluate the algorithms for a variant of the problem where a positive constant $\varepsilon$ has been added to the diagonal, corresponding to the PDE \begin{equation}\label{diffusion_with_enhanced_diagonal}
-\nabla(g\cdot\nabla u)+\varepsilon u = f.
\end{equation}
This test is relevant, in particular, to time-dependent parabolic PDE, where the diagonal term stems from discretization of the time derivative. For this experiment, we trained a second network, following the same training procedure as before, where for the training instances we used $\varepsilon h^2=10^{-8}$.
Figure \ref{epsilon_v_w_cycle} indicates that the network based algorithm retains its advantage in those kind of problems also, and is able to perform well on different values of $\varepsilon h^2$.

\begin{figure}[ht]
\begin{center}
\centerline{\includegraphics[width=0.7\columnwidth]{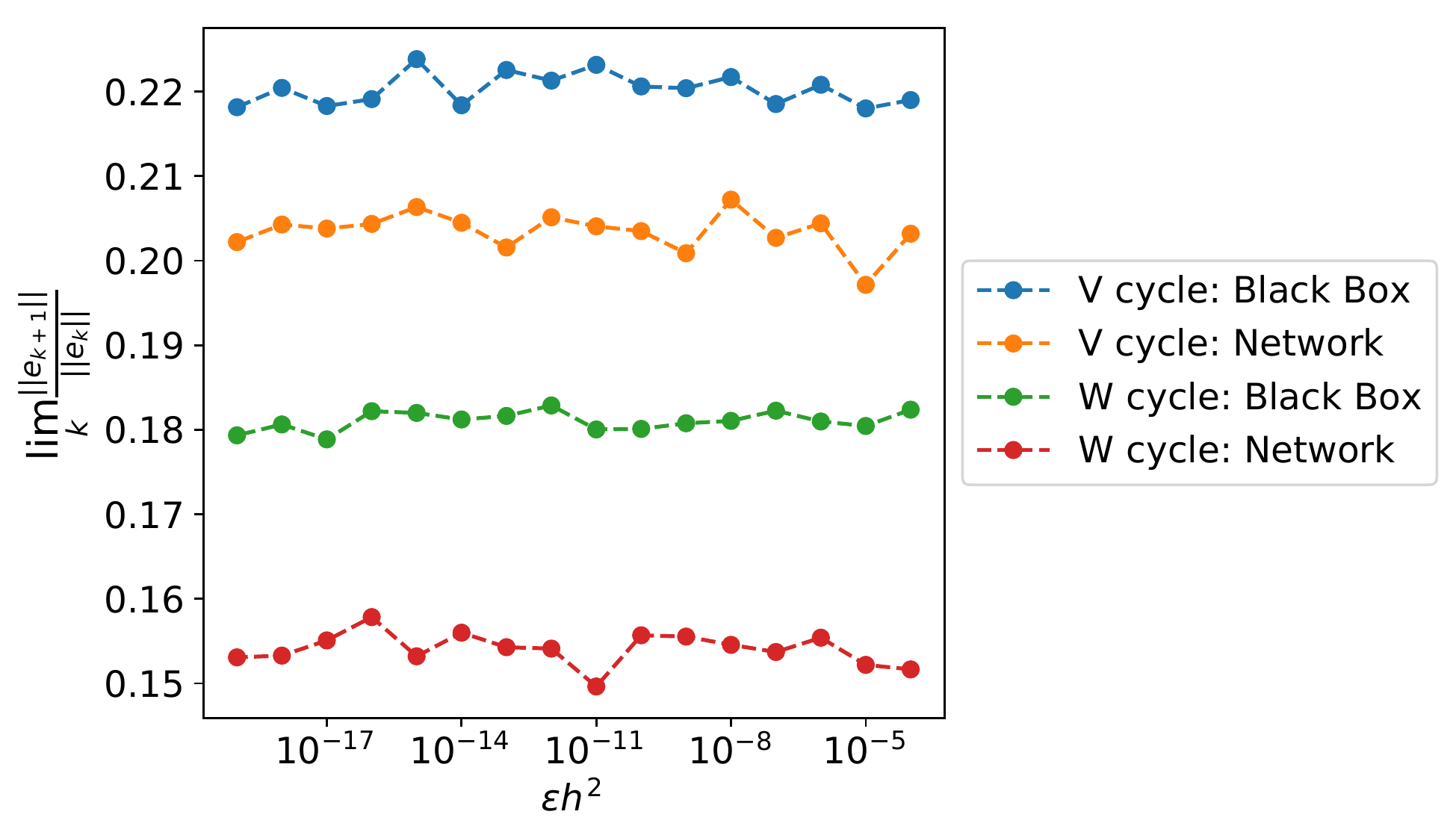}}
\caption{{\small Experiments with varying values of $\varepsilon h^{2}$ added to the diagonal. The graphs show the asymptotic error norm reduction factor of the V cycle and W cycles  per iteration, averaged over 100 experiments with grid size $256 \times 256$ ($h^2 = 1/65536$).}}
\label{epsilon_v_w_cycle}
\end{center}
\vskip -0.2in
\end{figure}

\section{Conclusion}
\label{Conclusion}
In this work we introduced a framework for devising multigrid solvers for parametric families of PDEs. Posed as a learning problem, this task is approached by learning a single mapping from discretization matrices to prolongation operators, using an efficient and unsupervised learning procedure. Experiments on 2D diffusion equations show improved convergence rates compared to the classical Black Box scheme, which has withstood the test of time for decades. Moreover, the experiments show generalization properties with respect to the problem size, boundary conditions and to some extent, its underlying distribution. Extending our work to triangulated and unstructured grids is an exciting direction we intend to pursue, as well as exploring simpler regression models which will allow for faster inference.

\bibliography{main}
\bibliographystyle{icml2019}

\appendix
\section{Appendix: Block Fourier analysis}
\label{Appendix}
Below we prove Theorem 1. The proof is based on two supporting lemmas. We begin with some mathematical terms.

Consider the $n\times n$  block-circulant matrix of the following form, where $n=kb$  and all numbering of rows, columns, blocks, etc., starts from 0 for convenience\begin{equation*}
    A=\left(\begin{array}{c}
A^{\left(0\right)}\\
A^{\left(1\right)}\\
A^{\left(2\right)}\\
\vdots\\
A^{\left(b-1\right)}
\end{array}\right),
\end{equation*}
where the blocks $A^{\left(m\right)},m=0,\dots,b-1$ are $k\times n$ real or complex matrices whose elements satisfy
\begin{equation} \label{block_property}
    A^{\left(m\right)}_{l,j} = A_{l,\mod\left(j-k,n\right)}^{\left(m-1\right)}, \  m=1,\dots,b-1
\end{equation}
and hence $A_{l,j}=A_{\mod(l-k,n),\mod(j-k,n)}$. Here, we are adopting the MATLAB form mod$(x,y) =$  ``$x$ modulo $y$'', i.e., the remainder obtained when dividing integer $x$ by integer $y$. Below, we continue to use $l$ and $j$ to denote row and column numbers, respectively, and apply the decomposition:
\begin{equation} \label{eq:Decomposition}
l = l_0 + tk,~~~~j = j_0 + sk \, ,
\end{equation}
where $l_0 = {\rm mod}(l,k)$, $t = \lfloor \frac{l}{k} \rfloor$, $j_0 = {\rm mod}(j,k)$, $s = \lfloor \frac{j}{k} \rfloor$. Note that $ l,j \in \{0,...,n-1\}$; $l_0,j_0 \in \{0,...,k-1\}$; $t,s \in \{0,...,b-1\} \, .$

\noindent
Let the column vector
$$
v_m = \left[ 1,e^{i\frac{2\pi m}{n}}, \ldots, e^{i\frac{2\pi mj}{n}}, \ldots, e^{i\frac{2\pi m(n-1)}{n}} \right]^*
$$
denote the unnormalized $m$th Fourier component of dimension $n$, for $m=0,\ldots,n-1$. Let $W$ denote the $n \times n$ matrix whose nonzero values are comprised of the elements of the first $b$ Fourier components as follows:
\begin{equation} \label{eq:W}
W_{l,j} = 
\delta_{l_0,j_0} v_s(l) \, ,
\end{equation}
where $v_s(l)$ denotes the $l$th element of $v_s$, and $\delta$ is the Kronecker delta. An example for $W$, with $k = 3$ and $b = 4$, is given in Fig. \ref{fig:W}.
\setcounter{MaxMatrixCols}{20}
\begin{figure*}[t]
\begin{equation*}
W=
\begin{bmatrix}
1&0&0&1&0&0&1&0&0&1&0&0 \\
0&1&0&0&e^{-i\frac{2\pi}{12}}&0&0&e^{-i\frac{4\pi}{12}}&0&0&e^{-i\frac{6\pi}{12}}&0 \\
0&0&1&0&0&e^{-i\frac{4\pi}{12}}&0&0&e^{-i\frac{8\pi}{12}}&0&0&e^{-i\frac{12\pi}{12}} \\
1&0&0&e^{-i\frac{6\pi}{12}}&0&0&e^{-i\frac{12\pi}{12}}&0&0&e^{-i\frac{18\pi}{12}}&0&0 \\
0&1&0&0&e^{-i\frac{8\pi}{12}}&0&0&e^{-i\frac{16\pi}{12}}&0&0&e^{-i\frac{24\pi}{12}}&0 \\
0&0&1&0&0&e^{-i\frac{10\pi}{12}}&0&0&e^{-i\frac{20\pi}{12}}&0&0&e^{-i\frac{30\pi}{12}} \\
1&0&0&e^{-i\frac{12\pi}{12}}&0&0&e^{-i\frac{24\pi}{12}}&0&0&e^{-i\frac{36\pi}{12}}&0&0 \\
0&1&0&0&e^{-i\frac{14\pi}{12}}&0&0&e^{-i\frac{28\pi}{12}}&0&0&e^{-i\frac{42\pi}{12}}&0 \\
0&0&1&0&0&e^{-i\frac{16\pi}{12}}&0&0&e^{-i\frac{32\pi}{12}}&0&0&e^{-i\frac{48\pi}{12}} \\
1&0&0&e^{-i\frac{18\pi}{12}}&0&0&e^{-i\frac{36\pi}{12}}&0&0&e^{-i\frac{54\pi}{12}}&0&0 \\
0&1&0&0&e^{-i\frac{20\pi}{12}}&0&0&e^{-i\frac{40\pi}{12}}&0&0&e^{-i\frac{60\pi}{12}}&0 \\
0&0&1&0&0&e^{-i\frac{22\pi}{12}}&0&0&e^{-i\frac{44\pi}{12}}&0&0&e^{-i\frac{66\pi}{12}}  
\end{bmatrix}
\end{equation*}
\caption{An example for $W$ with $k = 3$ and $b = 4$.}
\label{fig:W}
\end{figure*}

\begin{lemma}
$\frac{1}{\sqrt{b}}W$ is a unitary matrix.
\end{lemma}

\begin{proof}
Let $W_j$ and $W_m$ denote the $j$th and $m$th columns of $W$. Consider the inner product $W^\star_jW_m=\sum_{q=0}^{n-1}W^\star_j(q)W_m(q)$. For $\mod(j-m,k)\ne 0$, the product evidently vanishes because in each term of the sum at least one of the factors is zero. For $j=m$, the terms where $\mod(q,k)=j_0$ are equal to 1, while the rest are equal to zero, and therefore the product is $b$. Finally, for $j\ne m$ but $\mod(j-m,k)=0$, we can write $m=j+rk$ for some integer $r$ s.t. $0<|r|<b$. Summing up the non-zero terms, we obtain:
\begin{equation*}
\begin{split}
&\sum_{p=0}^{n-1}W_{j}^{\star}\left(p\right)W_{m}\left(p\right) =\sum_{q=0}^{b-1}v_{s}^{\star}\left(j_{0}+qk\right)v_{s+r}\left(j_{0}+qk\right)
\\ &=\sum_{q=0}^{b-1}e^{-i\frac{2\pi r\left(j_{0}+qk\right)}{n}}
=e^{-i\frac{2\pi rj_{0}}{n}}\sum_{q=0}^{b-1}\left(e^{-i\frac{2\pi r}{b}}\right)^{q}
\\ &=e^{-i\frac{2\pi rj_{0}}{n}}\frac{1-e^{-i\frac{2\pi rb}{b}}}{1-e^{-i\frac{2\pi r}{b}}}=0.
\end{split}
\end{equation*}
We conclude that $\frac{1}{\sqrt{b}}W^\star \frac{1}{\sqrt{b}}W=I_n$, the $n\times n$ identity matrix.
\end{proof}

\begin{lemma}
The similarity transformation, $\hat{A}=\frac{1}{\sqrt{b}}W^\star A \frac{1}{\sqrt{b}}W$, yields a block-diagonal matrix $\hat{A}$ with $b$ blocks of size $k\times k$.
\end{lemma}

\begin{proof}
Denote the $l$th row of $A$ by $A^l$. Then, the product $A^lW_j$ reads
\begin{equation*}
\begin{split}
    &A^{l}W_{j}=\sum_{p=0}^{n-1}A^{l}\left(p\right)W_{j}\left(p\right)=\sum_{q=0}^{b-1}A^{l}\left(j_{0}+qk\right)v_{s}\left(j_{0}+qk\right)\\
    &=\sum_{q=0}^{b-1}A^{l_{0}+tk}\left(j_{0}+qk\right)e^{-is\frac{2\pi\left(j_{0}+qk\right)}{n}}.
\end{split}
\end{equation*}
By repeated use of (\ref{block_property}), this yields
\begin{equation*}
\begin{split}
&A^{l}W_{j}=\sum_{q=0}^{b-1}A^{l_{0}}\left(j_{0}+\mathrm{mod}\left(q-t,b\right)k\right)e^{-is\frac{2\pi\left(j_{0}+qk\right)}{n}} \\
&=e^{-is\frac{2\pi tk}{n}}\sum_{q=0}^{b-1}A^{l_{0}}\left(j_{0}+\mathrm{mod}\left(q-t,b\right)k\right)e^{-is\frac{2\pi\left(j_{0}+\left(q-t\right)k\right)}{n}} \\
&=e^{-is\frac{2\pi t}{b}}\sum_{q=0}^{b-1}A^{l_{0}}\left(j_{0}+\mathrm{mod}\left(q-t,b\right)k\right)e^{-is\frac{2\pi\left(j_{0}+\mathrm{mod}\left(q-t,b\right)k\right)}{n}} \\
&=e^{-is\frac{2\pi t}{b}}A^{l_{0}}W_{j}.
\end{split}
\end{equation*}
Denoting $u_j=A^{(0)}W_j$, we thus obtain
\begin{equation*}
    AW_{j}=\left(\begin{array}{c}
u_{j}\\
e^{-i\frac{2\pi s}{b}}u_{j}\\
e^{-2i\frac{2\pi s}{b}}u_{j}\\
\vdots\\
e^{-\left(b-1\right)i\frac{2\pi s}{b}}u_{j}
\end{array}\right),
\end{equation*}
where the $q$th element of $u_j$, $q=0,\dots, k-1$, is given by
\begin{equation*}
    u_j(q)=A^qWj=\sum_{q=0}^{b-1}A^{q}\left(j_{0}+qk\right)v_{s}\left(j_{0}+qk\right).
\end{equation*}
Multiplying on the left by $W_l^\star$ for any $l=l_0+tk$ with $l_0 \in \{0,\dots,k-1\}$ and $t\in \{0,\dots,b-1\}$, yields
\begin{equation*}
\begin{split}
&W_{l}^{\star}AW_{j}    =\sum_{q=0}^{b-1}W_{l}^{\star}\left(l_{0}+qk\right)e^{-qi\frac{2\pi s}{b}}u_{j}\left(l_{0}\right) \\
&=\sum_{q=0}^{b-1}v_{t}^{\star}\left(l_{0}+qk\right)e^{-qi\frac{2\pi s}{b}}u_{j}\left(l_{0}\right) \\
&=\sum_{q=0}^{b-1}e^{i\frac{2\pi t\left(l_{0}+qk\right)}{n}}e^{-qi\frac{2\pi s}{b}}u_{j}\left(l_{0}\right) \\
&=e^{i\frac{2\pi tl_{0}}{n}}u_{j}\left(l_{0}\right)\sum_{q=0}^{b-1}e^{qi\frac{2\pi\left(t-s\right)}{b}}.
\end{split}
\end{equation*}
For $t\ne s$, the final sum yields
\begin{equation*}
    \sum_{q=0}^{b-1}e^{qi\frac{2\pi\left(t-s\right)}{b}}=\sum_{q=0}^{b-1}\left(e^{qi\frac{2\pi\left(t-s\right)}{b}}\right)^{q}=\frac{1-e^{i2\pi\left(t-s\right)}}{1-e^{i\frac{2\pi\left(t-s\right)}{b}}}=0.
\end{equation*}
We conclude that $W_l^\star A W_j$ vanishes unless $t=s$, which implies the block-periodic form stated in the proposition.
\end{proof}
For $t=s$, all the terms in the final sum in the proof are equal to 1, and therefore the sum is equal to $b$. This yields the following.
\begin{theorem}
Let $W$ be the matrix defined in (\ref{eq:W}). Then, $\hat{A}=\frac{1}{\sqrt{b}}W^\star A \frac{1}{\sqrt{b}}W=\mathrm{blockdiag}(B^{(0)},\dots,B^{(b-1)})$, where the elements of the $k\times k$ blocks $B^{(s)}$, $s=0,\dots,b-1,$ are given by
\begin{equation*}
\begin{split}
B_{l_{0},j_{0}}^{\left(s\right)}&=e^{i\frac{2\pi sl_{0}}{n}}u_{j}\left(l_{0}\right) \\
&=e^{i\frac{2\pi sl_{0}}{n}}\sum_{q=0}^{b-1}A^{l_{0}}\left(j_{0}+qk\right)v_{s}\left(j_{0}+qk\right) \\
&=e^{i\frac{2\pi sl_{0}}{n}}\sum_{q=0}^{b-1}A^{l_{0}}\left(j_{0}+qk\right)e^{-i\frac{2\pi s\left(j_{0}+qk\right)}{n}}\\
&=e^{i\frac{2\pi s\left(l_{0}-j_{0}\right)}{n}}\sum_{q=0}^{b-1}A^{l_{0}}\left(j_{0}+qk\right)e^{-i\frac{2\pi sq}{b}}.
\end{split}
\end{equation*}
\end{theorem}


\begin{remark}
The block Fourier analysis is applicable to discretized partial differential equations of any dimension $d$  by recursion. That is, for $d > 1$ the blocks of $A$ are themselves block-circulant, and so on. Remark 1 also generalizes to any dimension. That is, if the diameter of the discretization stencil is at most $k$ then each element of $B$ is easily computed from a single element of $A$.
\end{remark}
\end{document}